\documentclass{amsart}
\usepackage{fullpage}
\usepackage{stmaryrd}
\usepackage{amssymb}
\usepackage{mathrsfs}

\newtheorem{theorem}{Theorem}[section]
\newtheorem{lemma}[theorem]{Lemma}

\newtheorem{prop}[theorem]{Proposition}

\newtheorem*{thmA}{Theorem A}
\newtheorem*{thmB}{Theorem B}

\theoremstyle{definition}

\theoremstyle{remark}

\numberwithin{equation}{section}

\newcommand{\bC}{{\mathbf{C}}}

\newcommand{\bF}{{\mathbf{F}}}

\newcommand{\bN}{{\mathbf{N}}}

\newcommand{\Stab}{{\operatorname{Stab}}}
\newcommand{\Irr}{{\operatorname{Irr}}}

\newcommand{\Syl}{{\operatorname{Syl}}}

\newcommand{\dl}{{\operatorname{dl}}}
\newcommand{\cl}{{\operatorname{cl}}}

\newcommand{\Out}{{\operatorname{Out}}}
\newcommand{\Alt}{{\operatorname{Alt}}}

\let\nor=\triangleleft

\begin{document}

\title{Blocks of defect of $p$-solvable groups}

\author{YONG YANG}
\address{Department of Mathematics, Texas State University, 601 University Drive, San Marcos, TX 78666, USA.}
\makeatletter
\email{yang@txstate.edu}
\makeatother

\date{}

%\dedicatory{}

% at present the "communicated by" line appears only in ERA, PROC and JAG
%\commby{}

\begin{abstract}
Let $p$ be a prime such that $p \geq 5$. Let $G$ be a finite $p$-solvable group and let $p^a$ be the largest power of $p$ dividing $\chi(1)$ for an irreducible character $\chi$ of $G$, we show that $|G:\bF(G)|_p \leq p^{5.5a}$. Let $G$ be a finite $p$-solvable group with trivial maximal normal solvable subgroup and we denote $|G|_p=p^n$, then $G$ contains a block of defect less than or equal to $\lfloor \frac {2n} {3} \rfloor$. %Also, we obtain the best known bound for Huppert's $\rho-\sigma$ conjecture, $|\rho(G)| \leq 3\sigma(G)+2$ for $G$ solvable.
\end{abstract}

\maketitle
\maketitle
\Large
%%%%%%%%%%%%%%%%%%%%%%%%%%%%%%%%%%%%%%%%%%%%%%%%%%%%%%%%%%%%%%%%%%%%%%%%%
\section{Introduction} \label{sec:introduction8}

Let $G$ be a finite group. Let $p$ be a prime and $|G|_p=p^n$. An irreducible ordinary character of $G$ is called $p$-defect zero if and only if its degree is divisible by $p^n$. It is an interesting problem to give necessary and sufficient conditions for the existence of $p$-blocks of defect zero. If a finite group $G$ has a character of $p$-defect zero, then $O_p(G)=1$ ~\cite[Corollary 6.9]{WF}. Unfortunately, the converse is not true.

Although the block of defect zero may not exist in general, one could try to find the smallest defect $d(B)$ of a block $B$ of $G$. In ~\cite[Theorem A]{AENA2}, Espuelas and Navarro raised the following question. If $G$ is a finite group with $O_p(G)=1$ for some prime $p \geq 5$, and denote $|G|_p=p^n$, does $G$ contain a block of defect less then or equal to$\lfloor \frac n 2 \rfloor$?

%One of the results along this line is proved by Espuelas and Navarro ~\cite[Theorem A]{AENA2}. Let $G$ be a (solvable) group of odd order such that $O_p(G) = 1$ and $|G|_p = p^n$, then $G$ contains a $p$-block $B$ such that $d(B) \leq \lfloor \frac n 2 \rfloor$. The bound is best possible, as shown by an example in ~\cite{AENA2}.

%Using Theorem A, we prove the following result as a partial answer to this question. %The bound we obtain here is pretty sharp since $\lfloor \frac n 2 \rfloor$ is the best one may get.
In ~\cite{YY5}, the author provided a partial answer to the previous question under the condition where $G$ is solvable. In this paper, we prove a related result about $p$-solvable groups.%Let $G$ be a finite solvable group, let $p$ be a prime such that $p \geq 5$ and $O_p(G)=1$, and we denote $|G|_p=p^n$. Then $G$ contains a $p$-block $B$ such that $d(B) \leq \lfloor \frac {3n} {5} \rfloor $.

\begin{thmA}
Let $p$ be a prime such that $p \geq 5$. Let $G$ be a finite $p$-solvable group such that $O_{\infty}(G)=1$, and we denote $|G|_p=p^n$. Then $G$ contains a $p$-block $B$ such that $d(B) \leq \lfloor \frac {2n} {3} \rfloor $.
\end{thmA}

Let $p^a$ denote the largest power of $p$ dividing $\chi(1)$ for an irreducible character $\chi$ of $G$. Moret\'o and Wolf ~\cite[Theorem A]{MOWOLF} proved that for $G$ solvable, there exists a product $\theta=\chi_1(1) \cdots \chi_t(1)$ of distinct irreducible characters $\chi_i$ such that $|G: \bF(G)|$ divides $\theta(1)$ and $t \leq 19$. This implies that $|G:\bF(G)|_p \leq p^{19a}$. In this paper, we show the following result for $p$-solvable groups.

\begin{thmB}
Let $G$ be a $p$-solvable group where $p$ is a prime and $p \geq 5$. Suppose that $p^{a+1}$ does not divide $\chi(1)$ for all $\chi \in \Irr(G)$, then $|G: \bF(G)|_p\leq p^{5.5a}$.
\end{thmB}

We first fix some notation:
\begin{enumerate}

%\item If $V$ is a finite vector space of dimension $n$ over $\GF(q)$, where $q$ is a prime power, we denote by $\Gamma(q^n)=\Gamma(V)$ the semi-linear group of $V$, i.e.,
%\[\Gamma(q^n)=\{x \mapsto ax^{\sigma}\ |\ x \in \GF(q^n), a \in \GF(q^n)^{\times}, \sigma \in \Gal(\GF(q^n)/\GF(q))\},\] and we define \[\Gamma_0(q^n)=\{x \mapsto ax\ | \ x \in \GF(q^n), a \in \GF(q^n)^{\times}\}.\]

\item We use $\bF(G)$ to denote the Fitting subgroup of $G$. Let $\bF_0(G) \leq \bF_1(G) \leq \bF_2(G) \leq \cdots \leq \bF_n(G)=G$ denote the ascending Fitting series, i.e. $\bF_0(G)=1$, $\bF_1(G)=\bF(G)$ and $\bF_{i+1}(G)/\bF_i(G)=\bF(G/\bF_i(G))$. $\bF_i(G)$ is the $i$th ascending Fitting subgroup of $G$.

\item We use $F^*(G)$ to denote the generalized Fitting subgroup of $G$.

\item We use $O_{\infty}(G)=1$ to denote the maximal normal solvable subgroup of $G$.

\item Let $G$ be a finite group, we denote $cd(G)=\{\chi(1)\ | \ \chi \in \Irr(G) \}$.
%\item Let $\pi_0$ be the set of all the primes except $2$ and $3$.

%\item Let $G$ be a group acting on a set $\Omega$. We call an orbit $\mathscr{O}$ of $G$ on $\Omega$ to be $\pi_0$-regular if $|\mathscr{O}|_{\pi_0}=|G|_{\pi_0}$.
\end{enumerate}

\section{Blocks of small defect} \label{sec:Blocks}
Let $G$ be a finite group. Let $p$ be a prime and $|G|_p=p^n$. An irreducible ordinary character of $G$ is called $p$-defect $0$ if and only if its degree is divisible by $p^n$. By ~\cite[Theorem 4.18]{WF}, $G$ has a character of $p$-defect $0$ if and only if $G$ has a $p$-block of defect $0$. An important question in the modular representation theory of finite groups is to find the group-theoretic conditions for the existence of characters of $p$-defect $0$ in a finite group.

It is an interesting problem to give necessary and sufficient conditions for the existence of $p$-blocks of defect zero. If a finite group $G$ has a character of $p$-defect $0$, then $O_p(G)=1$ ~\cite[Corollary 6.9]{WF}. Unfortunately, the converse is not true. Zhang ~\cite{Zhang} and Fukushima ~\cite{Hiroshi1, Hiroshi2} provided various sufficient conditions where a finite group $G$ has a block of defect zero.

%Brauer¡¯s Problem 19, one of many conjectures and problems posed in [3], asks for a description of the number of defect zero p-blocks for a finite group in terms of its invariants. In [20], Robinson solved this problem; however it is difficult to determine his invariants for many groups.

%From the theory of modular representations of finite groups, we know [10,6.1.18] that an ordinary irreducible representation of a finite group $G$ is $p$-modularly irreducible and has defect zero if and only if the power of $p$ dividing the degree of the representation is equal to the power of $p$ dividing $|G|$.

%Results on regular orbit theorems have associated conditions for the existence of blocks of defect zero ~\cite{AE11}. For groups of odd order, with $O_p(G)=1$, the exceptions are essentially nilpotent by supersolvable as ~\cite[Theorem 1]{AE11} shows us.

Although the block of defect zero may not exist in general, one could try to find the smallest defect $d(B)$ of a block $B$ of $G$. One of the results along this line is given by ~\cite[Theorem A]{AENA2}. In ~\cite{AENA2}, Espuelas and Navarro bounded the smallest defect $d(B)$ of a block $B$ of $G$ using the $p$-part of $G$. Using an orbit theorem ~\cite[Theorem 3.1]{AE1} of solvable linear groups of odd order, they showed the following result. Let $G$ be a (solvable) group of odd order such that $O_p(G) = 1$ and $|G|_p = p^n$, then $G$ contains a $p$-block $B$ such that $d(B) \leq \lfloor \frac n 2 \rfloor$. The bound is best possible, as shown by an example in ~\cite{AENA2}.

It is not true in general that there exists a block $B$ with $d(B) \leq \lfloor \frac n 2 \rfloor$, as $G = A_7 (p = 2)$ shows us. However, the counterexamples where only found for $p=2$ and $p=3$. By work of Michler and Willems ~\cite{Michler,Willems} every simple group except possibly the alternating group has a block of defect zero for $p \geq 5$. The alternating group case was settled by Granville and Ono in ~\cite{GranvilleOno} using number theory.

Based on these, the following question raised by Espuelas and Navarro ~\cite{AENA2} seems to be natural. If $G$ is a finite group with $O_p(G)=1$ for some prime $p \geq 5$, and denote $|G|_p=p^n$, does $G$ contain a block of defect less then $\lfloor \frac n 2 \rfloor$?

In this section, we study this question and show that for a $p$-solvable group $G$ where $O_{\infty}(G)=1$ and $p \geq 5$, $G$ contains a block of defect less than or equal to $\lfloor \frac {2n} 3 \rfloor$. %The proof relies on the previous orbit theorem (Theorem ~\ref{generalcasep}). The bound we obtain here is not far away from the best possible bound $\lfloor \frac n 2 \rfloor$. We restate Theorem B for convenience.

%Remark: Although the result for the solvable group case is satisfactory, the conjecture of Espuelas and Navarro for arbitrary finite groups is wide open.

We need the following results about simple groups.

\begin{lemma}\label{simplecoprime}
Let $A$ act faithfully and coprimely on a nonabelian simple group $S$. Then $A$ has at least $2$ regular orbits on $\Irr(S)$.
\end{lemma}
\begin{proof}
This is ~\cite[Proposition 2.6]{MOTIEP}.
\end{proof}

%Using the classification of finite simple groups and Deligne-Lusztig theory, we can describe the simple groups with few character degrees.

\begin{theorem}\label{chardegsimple}
Let $G$ be a non-abelian finite simple group. Then $|cd(G)| \geq 4$,
\end{theorem}
\begin{proof}
This is Theorem C of ~\cite{MalleMoreto}.
\end{proof}

%Let $G$ be a finite group and let $p$ be a prime such that $\text{O}_p(G)=1$. Let $Z=Z(E(G))$. By the proof of \cite[Lemma 3.15]{gls} and by \cite[Lemma 3.16]{gls} we know that $F(G/Z)=F(G)/Z$, $E(G/Z)=E(G)/Z$, $F^*(G/Z)=F^*(G)/Z$, and $C_{G/Z}(E(G/Z))=C_G(E(G))/Z$.\\

\begin{prop}\label{prop3}
Let $S$ be a nonabelian simple group, and let $p$ be a prime such that $p \geq 3$ and $p$ does not divide $|S|$. Suppose $V = S_1 \times \cdots \times S_n$ where $S_i \cong S$. Assume $G$ is a $p$-solvable group that acts faithfully on $V$ via automorphisms, and assume the action of $G$ transitively permutes the $S_i$'s.

\begin{enumerate}

\item Assume $p \geq 5$, then there exists $N \nor G$,  $N \subseteq \bF_2(G)$ and there exist four $G$-orbits with representatives $v_1$, $v_2$, $v_3$ and $v_4 \in \Irr(V)$ such that for any $P \in \Syl_p(G)$, we have $\bC_P(v_i) \subseteq N$ for $1 \leq i \leq 4$. Moreover, the Sylow $p$-subgroup of $N \bF(G)/\bF(G)$ is abelian. %and the Sylow $p$-subgroup of $N \cap \bF(G)$ are abelian.

\item Assume $p = 3$, then there exists $N \nor G$,  $N \subseteq \bF_3(G)$ and there exist four $G$-orbits with representatives $v_1$, $v_2$, $v_3$ and $v_4 \in \Irr(V)$ such that for any $P \in \Syl_p(G)$, we have $\bC_P(v_i) \subseteq N$ for $1 \leq i \leq 4$. Moreover, the Sylow $p$-subgroup of $N \bF_2(G)/\bF_2(G)$ is abelian.
\end{enumerate}

\end{prop}
\begin{proof}
 Clearly $G$ is embedded in $\Out(V) = \Out(S) \wr S_n$. Set $K := G \cap \Out(S)^n$, and thus $G/K$ is a permutation group on $n$ letters.

 Assume $n=1$, then the result follows from Theorem ~\ref{chardegsimple} and the fact that $\Out(S)$ has a normal series of the form $A \nor B \nor C$ where $A$ is abelian, $B/A$ is cyclic and $C/B \cong 1, S_2$ or $S_3$.

 Assume $n>1$, and we first assume that $G/K$ is primitive. Since $S$ has four characters of different degrees by Theorem ~\ref{chardegsimple}, we may denote them to be $\theta, \lambda, \chi$ and $\psi$ and we may assume that $\theta(1) > \lambda(1) > \chi(1) > \psi(1)$.

 Assume $p \nmid |G/K|$, we may choose $v_1=\theta^n, v_2=\lambda^n, v_3=\chi^n$ and $v_4=\psi^n$. It is Clear that $\theta(1)^n > \lambda(1)^n > \chi(1)^n > \psi(1)^n$, and $\Syl_p(\bC_G(v_i)) \subseteq \Out(S)^n$ for $1 \leq i \leq 4$. Since $\Out(S)$ has a normal series of the form $A \nor B \nor C$ where $A$ is abelian, $B/A$ is cyclic and $C/B \cong 1, S_2$ or $S_3$, the result is clear.

 Assume $p \mid |G/K|$ and $n \geq 5$, then since $G/K$ is $p$-solvable, we know that $\Alt(n) \not \leq G/K$. By \cite[Lemma 1]{SDperm}(b), we can find a partition $\Omega=\Omega_1 \cup \Omega_2 \cup \Omega_3 \cup \Omega_4$ such that $\Stab_{G/K}(\Omega_1) \cap \Stab_{G/K}(\Omega_2) \cap \Stab_{G/K}(\Omega_3) \cap \Stab_{G/K}(\Omega_4)$ is a $2$-group and $t_1$,$t_2$,$t_3$ and $t_4$ are not all the same. We denote $t_i=|\Omega_i|$, $1 \leq i \leq 4$. By re-indexing, we may assume that $t_1 \geq t_2 \geq t_3 \geq t_4$.

 Since we know that not all the $t_i$s are the same, WLOG, we must have one of the followings,
 \begin{enumerate}
 \item $t_1 > t_2 \geq t_3 \geq t_4$. In this case, we construct four irreducible characters
  \begin{enumerate}
   \item $\alpha=\prod_{i=1}^{n} \alpha_i$, where $\alpha_i=\theta_i$ if $i \in \Omega_1$, $\alpha_i=\lambda_i$ if $i \in \Omega_2$, $\alpha_i=\chi_i$ if $i \in \Omega_3$, $\alpha_i=\psi_i$ if $i \in \Omega_4$. %$\theta^{t_1} \lambda^{t_2} \chi^{t_3} \psi^{t_4}$
   \item $\beta=\prod_{i=1}^{n} \beta_i$, where $\beta_i=\lambda_i$ if $i \in \Omega_1$, $\beta_i=\theta_i$ if $i \in \Omega_2$, $\beta_i=\chi_i$ if $i \in \Omega_3$, $\beta_i=\psi_i$ if $i \in \Omega_4$. %$\lambda^{t_1} \theta^{t_2} \chi^{t_3} \psi^{t_4}$,
   \item $\gamma=\prod_{i=1}^{n} \gamma_i$, where $\gamma_i=\chi_i$ if $i \in \Omega_1$, $ \gamma_i=\theta_i$ if $i \in \Omega_2$, $ \gamma_i=\lambda_i$ if $i \in \Omega_3$, $ \gamma_i=\psi_i$ if $i \in \Omega_4$. %$\chi^{t_1} \theta^{t_2} \lambda^{t_3} \psi^{t_4}$
   \item $\delta=\prod_{i=1}^{n} \delta_i$, where $\delta_i=\psi_i$ if $i \in \Omega_1$, $\delta_i=\theta_i$ if $i \in \Omega_2$, $\delta_i=\lambda_i$ if $i \in \Omega_3$, $\delta_i=\chi_i$ if $i \in \Omega_4$. %$\psi^{t_1} \theta^{t_2} \lambda^{t_3} \chi^{t_4}$,
  \end{enumerate}
 Those four characters have different degrees since

 $\alpha(1)=\theta(1)^{t_1} \lambda(1)^{t_2} \chi(1)^{t_3} \psi(1)^{t_4} > \beta(1)=\lambda(1)^{t_1} \theta(1)^{t_2} \chi(1)^{t_3} \psi(1)^{t_4} >$\\
 $\gamma(1)= \chi(1)^{t_1} \theta(1)^{t_2} \lambda(1)^{t_3} \psi(1)^{t_4}> \delta(1)=\psi(1)^{t_1} \theta(1)^{t_2} \lambda(1)^{t_3} \chi(1)^{t_4}$.

 \item $t_1=t_2>t_3 \geq t_4$. In this case, we construct four irreducible characters
  \begin{enumerate}
   \item $\alpha=\prod_{i=1}^{n} \alpha_i$, where $\alpha_i=\theta_i$ if $i \in \Omega_1$, $\alpha_i=\lambda_i$ if $i \in \Omega_2$, $\alpha_i=\chi_i$ if $i \in \Omega_3$, $\alpha_i=\psi_i$ if $i \in \Omega_4$. %$\theta^{t_1} \lambda^{t_2} \chi^{t_3} \psi^{t_4}$,
   \item $\beta=\prod_{i=1}^{n} \beta_i$, where $\beta_i=\theta_i$ if $i \in \Omega_1$, $\beta_i=\chi_i$ if $i \in \Omega_2$, $\beta_i=\lambda_i$ if $i \in \Omega_3$, $\beta_i=\psi_i$ if $i \in \Omega_4$. %$\theta^{t_1} \chi^{t_2} \lambda^{t_3} \psi^{t_4}$
   \item $\gamma=\prod_{i=1}^{n} \gamma_i$, where $\gamma_i=\theta_i$ if $i \in \Omega_1$, $\gamma_i=\psi_i$ if $i \in \Omega_2$, $\gamma_i=\lambda_i$ if $i \in \Omega_3$, $\gamma_i=\chi_i$ if $i \in \Omega_4$. %$\theta^{t_1} \psi^{t_2} \lambda^{t_3} \chi^{t_4}$,
   \item $\delta=\prod_{i=1}^{n} \delta_i$, where $\delta_i=\chi_i$ if $i \in \Omega_1$, $\delta_i=\psi_i$ if $i \in \Omega_2$, $\delta_i=\lambda_i$ if $i \in \Omega_3$, $\delta_i=\theta_i$ if $i \in \Omega_4$. %$\chi^{t_1} \psi^{t_2} \lambda^{t_3} \theta^{t_4}$,
  \end{enumerate}
 Those four characters have different degrees since

 $\alpha(1)=\theta(1)^{t_1} \lambda(1)^{t_2} \chi(1)^{t_3} \psi(1)^{t_4} > \beta(1)=\theta(1)^{t_1} \chi(1)^{t_2} \lambda(1)^{t_3} \psi(1)^{t_4} >$\\
 $\gamma(1)=\theta(1)^{t_1} \psi(1)^{t_2} \lambda(1)^{t_3} \chi(1)^{t_4} > \delta(1)= \chi(1)^{t_1} \psi(1)^{t_2} \lambda(1)^{t_3} \theta(1)^{t_4}$.

 \item $t_1=t_2=t_3 > t_4$. In this case, we construct four irreducible characters
   \begin{enumerate}
     \item $\alpha=\prod_{i=1}^{n} \alpha_i$, where $\alpha_i=\theta_i$ if $i \in \Omega_1$, $\alpha_i=\lambda_i$ if $i \in \Omega_2$, $\alpha_i=\chi_i$ if $i \in \Omega_3$, $\alpha_i=\psi_i$ if $i \in \Omega_4$. %$\theta^{t_1} \lambda^{t_2} \chi^{t_3} \psi^{t_4}$,
     \item $\beta=\prod_{i=1}^{n} \beta_i$, where $\beta_i=\theta_i$ if $i \in \Omega_1$, $\beta_i=\lambda_i$ if $i \in \Omega_2$, $\beta_i=\psi_i$ if $i \in \Omega_3$, $\beta_i=\chi_i$ if $i \in \Omega_4$. %$\theta^{t_1} \lambda^{t_2} \psi^{t_3} \chi^{t_4}$,
     \item $\gamma=\prod_{i=1}^{n} \gamma_i$, where $\gamma_i=\theta_i$ if $i \in \Omega_1$, $\gamma_i=\chi_i$ if $i \in \Omega_2$, $\gamma_i=\psi_i$ if $i \in \Omega_3$, $\gamma_i=\lambda_i$ if $i \in \Omega_4$. %$\theta^{t_1} \chi^{t_2} \psi^{t_3} \lambda^{t_4}$,
     \item $\delta=\prod_{i=1}^{n} \delta_i$, where $\delta_i=\lambda_i$ if $i \in \Omega_1$, $\delta_i=\chi_i$ if $i \in \Omega_2$, $\delta_i=\psi_i$ if $i \in \Omega_3$, $\delta_i=\theta_i$ if $i \in \Omega_4$. %$\lambda^{t_1} \chi^{t_2} \psi^{t_3} \theta^{t_4}$,
   \end{enumerate}
 Those four characters have different degrees since

$\alpha(1)=\theta(1)^{t_1} \lambda(1)^{t_2} \chi(1)^{t_3} \psi(1)^{t_4}> \beta(1)=\theta(1)^{t_1} \lambda(1)^{t_2} \psi(1)^{t_3} \chi(1)^{t_4}>$ \\
$\gamma(1)=\theta(1)^{t_1} \chi(1)^{t_2} \psi(1)^{t_3} \lambda(1)^{t_4} > \delta(1)= \lambda(1)^{t_1} \chi(1)^{t_2} \psi(1)^{t_3} \theta(1)^{t_4}$.

\end{enumerate}

Assume $p=3$, $p \mid |G/K|$ and $n=4$. In this case, we construct four irreducible characters $\alpha=\theta \lambda  \chi \psi$, $\beta=\theta  \lambda  \psi \psi$, $\gamma=\theta \chi \psi \psi$ and $\delta=\lambda \chi \psi \psi$.

 Those four characters have different degrees since $\alpha(1)=\theta(1) \lambda(1) \chi(1) \psi(1)> \beta(1)=\theta(1)  \lambda(1)  \psi^2(1) > \gamma(1)=\theta(1) \chi(1) \psi^2(1) > \delta(1)= \lambda(1) \chi(1) \psi^2(1)$.

 It is clear that $\bC_{G/K}(\alpha)=\bC_{G/K}(\beta)=\bC_{G/K}(\gamma)=\bC_{G/K}(\delta)=1$.

 Assume $p=3$, $p \mid |G/K|$ and $n=3$. In this case, we construct four irreducible characters, $\alpha=\theta \lambda  \chi$, $\beta=\theta  \lambda  \psi$, $\gamma=\theta \chi \psi$ and $\delta=\lambda \chi \psi$.

 Those four characters have different degrees since $\alpha(1)=\theta(1) \lambda(1)  \chi(1)> \beta(1)=\theta(1)  \lambda(1)  \psi(1)> \gamma(1)=\theta(1) \chi(1) \psi(1) > \delta(1)= \lambda(1) \chi(1) \psi(1)$.

 It is clear that $\bC_{G/K}(\alpha),\bC_{G/K}(\beta),\bC_{G/K}(\gamma),\bC_{G/K}(\delta)$ is a $2$-group.

Thus, we may always find four irreducible characters $v_1=\alpha$, $v_2=\beta$, $v_3=\gamma$ and $v_4=\delta$ in $\Irr(V)$ of different degrees such that for any $P \in \Syl_p(G)$, we have $\bC_G(v_i) \subseteq \Out(S)^n$ for $1 \leq i \leq 4$. Since $\Out(S)$ has a normal series of the form $A \nor B \nor C$ where $A$ is abelian, $B/A$ is cyclic and $C/B \cong 1, S_2$ or $S_3$, the result is clear.

%If $p \mid |G/K|$, then since $G/K$ is $p$-solvable, we know that $n \geq p$ and $\Alt(n) \not \leq G/K$. By \cite[Lemma 1]{SDperm}, we may find four $v \in \Irr(V)$ of different degrees such that $\Syl_p(\bC_G(v)) \subseteq \Out(S)^n$.

Assume $n>1$, and we assume that $G/K$ is not primitive. The argument follows by induction.
\end{proof}

\begin{prop}\label{prop4}
Let $G$ be a finite $p$-solvable group where $O_{\infty}(G)=1$. Then $F^*(G)=E_1 \times \dots \times E_m$ is a product of $m$ finite non-abelian simple groups $E_i$, $1 \leq i \leq m$ permuted by $G$. Let $L=\bigcap_i \bN_G(E_i)$. Clearly $L/F^*(G) \leq \Out E_1 \times \dots \times \Out E_m$, and we denote $\bar G=G/F^*(G)$. %, then there exists $N \nor \bar G$ where $N \subseteq \bF_2(\bar G)$, and there exists $v \in \Irr(F^*(G))$ such that for any $P \in \Syl_p(\bar G)$, we have $\bC_P(v) \subseteq N$. Moreover, the Sylow $p$-subgroup of $N \bF(\bar G)/\bF(\bar G)$ is abelian.%and the Sylow $p$-subgroup of $N \cap \bF(\bar G)$ are abelian.
\begin{enumerate}
\item Assume $p \ge 5$, then there exists $N \nor \bar G$ where $N \subseteq \bF_2(\bar G)$, and there exists $v \in \Irr(F^*(G))$ such that for any $P \in \Syl_p(\bar G)$, we have $\bC_P(v) \subseteq N$. Moreover, the Sylow $p$-subgroup of $N \bF(\bar G)/\bF(\bar G)$ is abelian.%and the Sylow $p$-subgroup of $N \cap \bF(\bar G)$ are abelian.

\item Assume $p = 3$, then there exists $N \nor \bar G$ where $N \subseteq \bF_3(\bar G)$, and there exists $v \in \Irr(F^*(G))$ such that for any $P \in \Syl_p(\bar G)$, we have $\bC_P(v) \subseteq N$. Moreover, the Sylow $p$-subgroup of $N \bF_2(\bar G)/\bF_2(\bar G)$ is abelian.
\end{enumerate}
\end{prop}
\begin{proof}
 %$H = G/\bC_G(N)$ is embedded in $\Aut(N) = \Aut(S) \wr S_k$. Set $K := H \cap \Aut(S)^k$. Note that $N$ can be viewed as a subgroup of $H$ and therefore $N$ is a subgroup of $K$. Also, $H/K$ is a permutation group on $k$ letters.

%$E(G)/Z=E_1 \times E_2 \times \cdots \times E_n$ where $E(G)/Z$ is a direct product of simple groups.

Consider $K=G/F^*(G)$, $K$ acts faithfully on $F^*(G)$. Next, we group the simple groups in the direct product of $F^*(G)$ where $K$ acts transitively. We denote $F^*(G)=L_1 \times \cdots \times L_s$, $L_j=E_{j1} \times E_{j2} \times \cdots \times E_{jm_j}$, $1 \leq j \leq s$ where $K$ transitively permutes the simples groups inside the direct product of $L_j$. Clearly $E_{j1} \cong E_{j2} \cdots \cong E_{jm_j}$.

We see that $K$ can be embedded as a subgroup of $K/\bC_K(L_1) \times \cdots \times K/\bC_K(L_s)$ and we denote $K_i=K/\bC_K(L_i)$, $C_i=\bC_K(L_i)$.

%Clearly $E(G)/Z=L_1 \times L_2 \cdots \times L_t$ and we know that $K_{i-1}/K_i$ acts transitively and faithfully on $L_i$ where $i=1, \dots, t$.

If $p \geq 5$, by applying Proposition \ref{prop3}(1) to the action of $K_i$ on $L_i$, there exists $v_{i} \in \Irr(L_i)$, and $N_i \nor K_i$ such that for any $P_i \in \Syl_{p}(K_i)$, $\bC_{P_i}(v_{i}) \subseteq N_i \subseteq \bF_2(K_i)$. Also the Sylow $p$-subgroup of $N_i \bF(K_i)/\bF(K_i)$ is abelian. % and the $p$-subgroup of $N_i \cap \bF(K_i)$ are

Let $v=\sum v_{i}$ and $N= K \cap (\prod N_i)$. Let $P \in \Syl_{p}(G)$ and $P_i=P C_i/C_i$. $\bC_P(v) \subseteq \prod \bC_{P_i}(v_{i}) \subseteq \prod N_i$. Clearly $N \bF(K)/\bF(K) \subseteq \prod N_i \bF(K_i)/\bF(K_i)$ and the result follows.%$N \cap \bF(K) \subseteq \prod N_i \cap \bF(K_i)$.

If $p=3$, by applying Proposition \ref{prop3}(2) to the action of $K_i$ on $L_i$, there exists $v_{i} \in \Irr(L_i)$, and $N_i \nor K_i$ such that for any $P_i \in \Syl_{p}(K_i)$, $\bC_{P_i}(v_{i}) \subseteq N_i \subseteq \bF_3(K_i)$. Also the Sylow $p$-subgroup of $N_i \bF_2(K_i)/\bF_2(K_i)$ is abelian. % and the $p$-subgroup of $N_i \cap \bF(K_i)$ are

Let $v=\sum v_{i}$ and $N= K \cap (\prod N_i)$. Let $P \in \Syl_{p}(G)$ and $P_i=P C_i/C_i$. $\bC_P(v) \subseteq \prod \bC_{P_i}(v_{i}) \subseteq \prod N_i$. Clearly $N \bF_3(K)/\bF_3(K) \subseteq \prod N_i \bF_3(K_i)/\bF_3(K_i)$ and the result follows.

\end{proof}

\begin{thmA} \label{sdefect}
Let $p$ be a prime such that $p \geq 5$. Let $G$ be a finite $p$-solvable group such that $O_{\infty}(G)=1$, and we denote $|G|_p=p^n$. Then $G$ contains a $p$-block $B$ such that $d(B) \leq \lfloor \frac {2n} {3} \rfloor $.
\end{thmA}
\begin{proof}

Since $O_{\infty}(G)=1$, $F^*(G)=L_1 \times \dots \times L_m$ is the product of $m$ finite non-abelian simple groups permuted by $G/F^*(G)$. We denote $\bar G=G/F^*(G)$ and let $J=\bigcap_i \bN_G(L_i)$. Clearly $J/F^*(G) \leq \Out L_1 \times \dots \times \Out L_m$. By Proposition ~\ref{prop4}(1), there exists $N \nor \bar G$ where $N \subseteq \bF_2(\bar G)$ and one $\bar G$-orbit with representative $\lambda \in \Irr(F^*(G))$ such that for any $P \in \Syl_p(\bar G)$, we have $\bC_P(v) \subseteq N$. Furthermore, the Sylow $p$-subgroup of $N \bF(\bar G)/\bF(\bar G)$ is abelian. %and the Sylow $p$-subgroup of $N \cap \bF(\bar G)$ are abelian.

 Let $p^n=|\bar G|_p$, $p^{n_1}=|N \cap \bF(\bar G)|_p$, $p^{n_2}=|N \bF(\bar G) :\bF(\bar G)|_p$ and $p^{n_3}=|\bar G:N|_p$. Clearly $n=n_1+n_2+n_3$.

Take $\chi \in \Irr(G)$ lying over $\lambda$ and let $B$ be the $p$-block of $G$ containing $\chi$. As $F^*(G)$ is a $p'$-group, \cite[Lemma V.2.3]{WF} shows that every irreducible character $\psi$ in $B$ has $\lambda$ as an irreducible constituent. Now $\psi(1)_p \geq |\bar G: \bC_{\bar G}(\lambda)|_p  \geq p^{n_3}$.% by Theorem ~\ref{generalcasep}???.

Let $P/\bF(\bar G)$ be a Sylow $p$-subgroup of $N \bF(\bar G)/\bF(\bar G)$. Where $Y = O_{p'}(\bF(\bar G))$, observe that $W = \Irr(Y/\Phi(Y))$ is a faithful and completely reducible $P/\bF(\bar G)$-module. By Gow's regular orbit theorem ~\cite[2.6]{GOW}, we have $\mu \in W$ such that $\bC_P(\mu) = \bF(\bar G)$. We may view $\mu$ as a character of the preimage $X$ of $Y$ in $\bar G$. Observe that $X$ is a $p'$-group. Take $\chi \in \Irr(\bar G)$ lying over $\mu$, and let $B$ be the $p$-block of $G$ containing $\chi$. \cite[Lemma V.2.3]{WF} shows that every irreducible character $\psi$ in $B$ has $\lambda$ as an irreducible constituent. Now $\chi$ lies over an irreducible character $\psi$ of $P$ lying over $\mu$. Clearly, $\psi(1)_p \ge |N \bF(\bar G)/\bF(\bar G)|_p \ge p^{n_2}$. As $P$ is normal in $G$, we have $\chi(1)_p \ge \psi(1)_p$. %Now the same argument as in case (1) completes the proof.

Let $P_1$ be the Sylow $p$-subgroup of $\bF(\bar G) \cap N$. By Lemma ~\ref{simplecoprime}, we may find $\nu \in \Irr(F^*(G))$ such that $\bC_{P_1}(\nu)=1$. Thus by a similar argument as before, we may find a $p$-block $B$ of $G$ such that for every irreducible character $\psi$ in $B$, $\psi(1)_p \geq |N \cap \bF(\bar G)|_p=p^{n_1}$.

%We can choose $\phi \in \Irr(G)$ such that $|N \bF(\bar G): \bF(\bar G)|_{p}$ divides $\phi(1)$. We can choose  $\mu \in \Irr(G)$ such that $|N \cap \bF(\bar G)|_{p}$ divides $\mu(1)$.

%Let $P/F^*(\bar G)$ be the Sylow $p$-subgroup of $K F^*(\bar G)/F^*(\bar G)$ and let $P$ be the preimage of it. Let $Y =O_{p'}(F^*(\bar G))$, observe that $W = \Irr(Y/\Phi(Y))$ is a faithful and completely reducible $P/F^*(\bar G)$-module. Since $P/F^*(\bar G)$ is abelian, there exists $\mu \in W$ such that $C_P(\mu) = F^*(\bar G)$. We may view $\mu$ as a character of the preimage $X$ of $Y$ in $G$. Observe that $X$ is a $p'$-group. Take $\xi \in \Irr(G)$ lying over $\mu$. Now $\xi$ lies over an irreducible character $\phi$ of $P$ lying over $\mu$. Clearly, $\phi(1)_p \geq |P/F^*(\bar G)|=p^{n_2}$. As $P$ is normal in $G$, we have $\xi(1)_p \geq \phi(1)_p \geq p^{n_2}$. Let $B$ be the $p$-block of $G$ containing $\xi$. As $X$ is a $p'$-group, \cite[Lemma V.2.3]{WF} shows that every irreducible character $\delta$ in $B$ has $\mu$ as an irreducible constituent and $\delta(1)_p \geq p^{n_2}$.

%Let $P_1/F^*(G)$ be the Sylow $p$-subgroup of $K \cap F^*(\bar G)$ and let $P_1$ be the preimage of it. Since $P_1/F^*(G)$ is normal in $G/F^*(G)$, $V = \Irr(F^*(G))$ is a faithful and completely reducible $P_1/F^*(G)$-module. Since $P_1/F^*(G)$ is abelian, using a similarly argument as the previous paragraph, we may find a block $B$ such that every irreducible character $\varphi$ in $B$ satisfies $\varphi(1)_p \geq |K \cap F^*(\bar G)|_p=p^{n_1}$.

We know there is a block $B$ such that for every irreducible character $\alpha$ in $B$, $\alpha(1)_p \geq \max(p^{n_3}, p^{n_2}, p^{n_1})$. Since $n_1+n_2+n_3=n$, it is not hard to see that $\alpha(1)_p \geq p^ {\lceil \frac {n} {3} \rceil} $ and thus $d(B) \leq \lfloor \frac {2n} {3} \rfloor$.
\end{proof}

\section{$p$ part of $|G:\bF(G)|$, character degrees and conjugacy class sizes} \label{p part of G/F(G)}
%it is a consequence of work ~\cite{GLUCKWOLF} on the height-zero conjecture that the derived length of $P$ is at most $2a+1$ for $G$ $p$-solvable. However a $p$-group of derived length $2$ can have irreducible characters of arbitrarily large degree.
If $P$ is a Sylow $p$-subgroup of a finite group $G$ it is reasonable to expect that the degrees of irreducible characters of $G$ somehow restrict those of $P$. Let $p^a$ denote the largest power of $p$ dividing $\chi(1)$ for an irreducible character $\chi$ of $G$ and $b(P)$ denote the largest degree of an irreducible character of $P$. Conjecture $4$ of Moret\'o ~\cite{Moret1} suggested $\log b(P)$ is bounded as a function of $a$. Moret\'o and Wolf ~\cite{MOWOLF} have proven this for $G$ solvable and even something a bit stronger, namely the logarithm to the base of $p$ of the $p$-part of $|G: \bF(G)|$ is bounded in terms of $a$. In fact, they showed that $|G:\bF(G)|_p \leq p^{19a}$. Moret\'o and Wolf ~\cite{MOWOLF} also proved that $|G: \bF(G)|_p \leq p^{2a}$ for odd order groups, this can also be deduced from ~\cite{AENA2}. This bound is best possible, as shown by an example in ~\cite{AENA2}. %It is possible that $p^a < b(P)$ at least when $p=2$, as shown by an example of Isaacs ~\cite[Example 5.1]{Moret1}. %(Isaacs has also provided solvable examples for $p=3$)

In this paper, we show that for $p$-solvable groups where $p \geq 5$, $|G:\bF(G)|_p \leq p^{5.5a}$.

\begin{theorem}\label{chardegreeboundnew}
Let $G$ be a $p$-solvable group where $p$ is a prime and $p \geq 5$. Suppose that $p^{a+1}$ does not divide $\chi(1)$ for all $\chi \in \Irr(G)$ and let $P \in \Syl_p(G)$, then $|G: \bF(G)|_p\leq p^{5.5a}$, $b(P)\leq p^{6.5a}$ and $\dl(P) \leq \log_2 a + 5 + \log_2 6.5$.
\end{theorem}
\begin{proof}
Let $T=O_{\infty}(G)$, the maximal normal subgroup of $G$. Since $\bF(G) \subseteq T$, $\bF(T)=\bF(G)$. Since $T \nor G$, $p^{a+1}$ does not divide $\lambda(1)$ for all $\lambda \in \Irr(T)$. Thus by ~\cite[Remark of Corollary 5.3]{YY5}, $|T: \bF(G)|_p\leq p^{2.5a}$.

%By ~\cite[Theorem 3.3]{YY5}, we may choose $\chi \in \Irr(T)$ and $K \nor T$ such that $\bF(T)$ is not in $\Ker \chi$ and $|T: K|_{p}$ divides $\chi(1)$. We can choose $\phi \in \Irr(T)$ such that $\bF(T)$ is in $\Ker \phi$ and $|K \bF_2(T): \bF_2(T)|_{p}$ divides $\phi(1)$. We can choose $\mu \in \Irr(T)$ such that $\bF(T)$ is not in $\Ker \mu$ and $|K \cap \bF_2(T): \bF(T)|_{p}$ divides $\mu(1)$.

Let $\tilde G=G/O_{\infty}(G)$ and $\bar G=\tilde G / F^*(\tilde G)$. It is clear that $F^*(\tilde G)$ is a direct product of finite non-abelian simple groups. By Proposition ~\ref{prop4}(1), there exists $N \nor \bar G$ where $N \subseteq \bF_2(\bar G)$ such that for any $P \in \Syl_p(\bar G)$, we have $\bC_P(v) \subseteq N$, and the Sylow $p$-subgroup of $N \bF(\bar G)/\bF(\bar G)$ is abelian. It is clear that we may find $\bar \gamma \in \Irr(\bar G)$ such that $|\bar G: N|_{p}$ divides $\bar \gamma(1)$. %We can choose $\bar \phi \in \Irr(\bar G)$ such that $|N \bF(\bar G): \bF(\bar G)|_{p}$ divides $\bar \phi(1)$. % and the Sylow $p$-subgroup of $N \cap \bF(\bar G)$ are abelian.

%We can choose  $\bar \mu \in \Irr(\tilde G)$ such that $|N \cap \bF(\bar G)|_{p}$ divides $\bar \mu(1)$. %Then we have $\mu, \chi, \phi \in \Irr(\bF_2(G))$ such that $\mu^G$ is irreducible and $|\bF_3(G):\bF(G)|$ divides $\chi(1)\phi(1)$.

Let $P/\bF(\bar G)$ be a Sylow $p$-subgroup of $N \bF(\bar G)/\bF(\bar G)$. Where $Y = O_{p'}(\bF(\bar G))$, observe that $W = \Irr(Y/\Phi(Y))$ is a faithful and completely reducible $P/\bF(\bar G)$-module. By Gow's regular orbit theorem ~\cite[2.6]{GOW}, we have $\mu \in W$ such that $\bC_P(\mu) = \bF(\bar G)$. We may view $\mu$ as a character of the preimage $X$ of $Y$ in $\bar G$. Take $\bar \alpha \in \Irr(\bar G)$ lying over $\mu$, and $\bar \alpha$ lies over an irreducible character $\bar \psi$ of $P$ lying over $\mu$. Clearly, $\bar \psi(1)_p \ge |N \bF(\bar G)/\bF(\bar G)|_p$. As $P$ is normal in $G$, we have $\bar \alpha(1)_p \ge \bar \psi(1)_p \ge |N \bF(\bar G)/\bF(\bar G)|_p$.

 %Observe that $X$ is a $p'$-group. Take $\chi \in \Irr(\bar G)$ lying over $\mu$, and let $B$ be the $p$-block of $G$ containing $\chi$. \cite[Lemma V.2.3]{WF} shows that every irreducible character $\psi$ in $B$ has $\lambda$ as an irreducible constituent. Now $\chi$ lies over an irreducible character $\psi$ of $P$ lying over $\mu$. Clearly, $\psi(1)_p \ge |N \bF(\bar G)/\bF(\bar G)|_p \ge p^{n_2}$. As $P$ is normal in $G$, we have $\chi(1)_p \ge \psi(1)_p$. %Now the same argument as in case (1) completes the proof.

Let $P_1$ be the Sylow $p$-subgroup of $\bF(\bar G) \cap N$. By Lemma ~\ref{simplecoprime}, we may find $\nu \in \Irr(F^*(\tilde G))$ such that $\bC_{P_1}(\nu)=1$. Thus by a similar argument as before, we may find an irreducible character $\bar \beta$ of $\bar G$ such that $\bar \beta(1)_p \geq |N \cap \bF(\bar G)|_p$.

%Let $\gamma$ and $\delta$ in $\Irr(G)$ lie over $\chi$ and $\phi$.
%Then $\phi$ is distinct since $\bF(G)$ is in $\Ker \phi$ but not in $\Ker \chi$ and $\Ker \mu$. If $\mu$ is $\chi$, the product $\theta=\chi \phi$ satisfies the conclusion. Else $\theta=\chi \phi \mu$ does.

%Let $\gamma$ and $\delta$ in $\Irr(G)$ lie over $\chi$ and $\phi$.
%Then $\phi$ is distinct since $\bF(G)$ is in $\Ker \phi$ but not in $\Ker \chi$ and $\Ker \mu$. If $\mu$ is $\chi$, the product $\theta=\chi \phi$ satisfies the conclusion. Else $\theta=\chi \phi \mu$ does.

%There exists a product $\theta=\chi_1 \chi_2 \chi_3$ of distinct irreducible characters $\chi_i$ such that $|G: \bF(G)|_{\pi_0}$ divides $\theta(1)$ by Theorem ~\ref{pipartofGFG} and so

Thus $|G:\bF(G)|_p \leq p^{5.5a}$. If $P \in \Syl_p(G)$, then $b(P) \leq |P: O_p(G)||b(O_p(G))|=|G:\bF(G)|_p |b(O_p(G))| \leq p^{5.5a}p^a=p^{6.5a}$.

Now, we want to prove the last part of the statement. By ~\cite[Theorem 12.26]{Isaacs/book} and the nilpotency of $P$, we have that $P$ has an abelian subgroup $B$ of index at most $b(P)^4$. By ~\cite[Theorem 5.1]{Podoski}, we deduce that $P$ has a normal abelian subgroup $A$ of index at most $|P:B|^2$. Thus, $|P:A| \leq |P:B|^2 \leq b(P)^{8s}$, where $b(P)=p^s$. By ~\cite[Satz III.2.12]{Huppert1}, $\dl(P/A) \leq 1+\log_2(8s)$ and so $\dl(P) \leq 2+ \log_2(8s)=5+\log_2(s)$. Since $s$ is at most $6.5a$, the result follows.
\end{proof}

%Remark, the bound $p^{3a}$ can be easily improved to $p^{2.5a}$ using the same technique in the proof of Theorem B, and we give a sketch here. One may write $|G:\bF(G)|_p=p^{a_1+b_1+c_1}$, where there exist $\chi_1$ such that $p^{a_1} \mid |\chi_1(1)|_p$, $\chi_2$ such that $p^{b_1} \mid |\chi_2(1)|_p$ and $\chi_3$ such that $p^{c_1+\frac 1 2 a_1} \mid |\chi_3(1)|_p$. If $a_1 \geq \max(b_1,c_1+\frac 1 2 a_1)$, then $a_1+b_1+c_1 \leq 2.5 a_1$. If $b_1 \geq \max(a_1, c_1+\frac 1 2 a_1)$, then $a_1+b_1+c_1 \leq 2.5 b_1$. If $c_1+\frac 1 2 a_1 \geq \max(a_1,b_1)$, then $a_1+b_1+c_1 \leq 2.5 (c_1+\frac 1 2 a_1)$. \\

We now state the conjugacy analogs of Theorem ~\ref{chardegreeboundnew}. Given a group $G$, we write $b^*(G)$ to denote the largest size of the conjugacy classes of $G$. %The following result improves ~\cite[Corollary B']{MOWOLF} for $p \geq 5$.
\begin{theorem}\label{conjugacybound}
Let $G$ be a $p$-solvable group where $p$ is a prime and $p \geq 5$. Suppose that $p^{a+1}$ does not divide $|C|$ for all $C \in \cl(G)$ and let $P \in \Syl_p(G)$, then $|G: \bF(G)|_p \leq p^{5.5a}$, $b^*(P)\leq p^{6.5a}$ and $|P'| \leq p^{6.5a(6.5a+1)/2}$.
\end{theorem}
\begin{proof}
The proof of the first statement goes similarly as the previous one. Write $N=O_p(G)$. It is clear that $|N: \bC_N(x)|$ divides $|G: \bC_G(x)|$ for all $x \in G$. Thus, if we take $x \in P$ we have that
\[|\cl_P(x)|= |P: \bC_P(x)| \leq |P:N||N: \bC_N(x)| \leq p^{5.5a} p^a=p^{6.5a}\]
 Finally, to obtain the bounds for the order of $P'$ is suffices to apply a theorem of Vaughan-Lee \cite[Theorem VIII.9.12]{Huppert2}.
\end{proof}

The following is a corrected version of ~\cite[Theorem 5.1]{YY1} (Note the $t \leq 19$ in the original statement should be  $t \leq 15$).
\begin{theorem}\label{correction}
If $G$ is solvable, there exists a product $\theta=\chi_1 \dots \chi_t$ of distinct irreducible characters $\chi_i$ of $G$ such that $|G:\bF(G)|$ divides $\theta(1)$ and $t \leq 15.$ Furthermore, if $|\bF_8(G)|$ is odd then we can take $t \leq 3$ and if $|G|$ is odd we can take $t \leq 2$.
\end{theorem}

For the case $p=3$, we have the following results. The proof is similar as before but using Theorem ~\ref{correction} instead of ~\cite[Remark of Corollary 5.3]{YY5}, and Proposition ~\ref{prop4}(2) instead of Proposition ~\ref{prop4}(1).

\begin{theorem}\label{chardegreeboundnew3}
Let $G$ be a $p$-solvable group where $p$ is a prime and $p = 3$. Suppose that $p^{a+1}$ does not divide $\chi(1)$ for all $\chi \in \Irr(G)$ and let $P \in \Syl_p(G)$, then $|G: \bF(G)|_p\leq p^{20a}$, $b(P)\leq p^{21a}$ and $\dl(P) \leq \log_2 a + 5 + \log_2 21$.
\end{theorem}

\begin{theorem}\label{conjugacybound3}
Let $G$ be a $p$-solvable group where $p$ is a prime and $p = 3$. Suppose that $p^{a+1}$ does not divide $|C|$ for all $C \in \cl(G)$ and let $P \in \Syl_p(G)$, then $|G: \bF(G)|_p \leq p^{20a}$, $b^*(P)\leq p^{21a}$ and $|P'| \leq p^{21a(21a+1)/2}$.
\end{theorem}

\section{Discussions} \label{Discussions}

%We remark here that for Theorem ~\ref{chardegreeboundnew} and ~\ref{conjugacybound}, it is possible to find related bounds for $3$-solvable groups using similar arguments.

In order to improve the bounds of the results in this paper, one might need to study the situation when a $p$-solvable group acts on a field of characteristic does not equal to $p$, and hope that a similar result as ~\cite[Theorem 3.3]{YY5} holds. Also, a strengthened version of Lemma ~\ref{simplecoprime} would also be helpful in improving the bounds.

%\section{Acknowledgement} \label{sec:Acknowledgement}
%The author would like to thank for the financial support from the AMS-Simons travel grant.

%\section{Acknowledgement} \label{sec:Acknowledgement}
%I wish to thank Alexandre Turull for his constant encouragement. I am also greatly in debt to Thomas Keller for valuable discussions. Some of the work was done while I was visiting Texas State University-San Marcos. I thank the Mathematics Department for its hospitality.

%%%%%%%%%%%%%%%%%%%%%%%%%%%%%%%%%%%%%%%%%%%%%%%%%%%%%%%%%%%%%%%%%%%%%%%%%

\end{document}